\documentclass[12pt]{article}%
\usepackage{amsmath}
\usepackage{amsfonts}
\usepackage{amssymb}
\usepackage{graphicx}%
\makeatletter
\newfont{\footsc}{cmcsc10 at 8truept}
\newfont{\footbf}{cmbx10 at 8truept}
\newfont{\footrm}{cmr10 at 10truept}
\newtheorem{theorem}{Theorem}

\newtheorem{conjecture}[theorem]{Conjecture}

\newtheorem{lemma}[theorem]{Lemma}

\newtheorem{proposition}[theorem]{Proposition}

\newenvironment{proof}[1][Proof]{\noindent{\textbf {#1}  }}  {\hfill$\blacksquare$\bigskip}

\def\blfootnote{\xdef\@thefnmark{}\@footnotetext}
\begin{document}

\title{Maxima of the $Q$-index: graphs without long paths}
\author{Vladimir Nikiforov\thanks{Department of Mathematical Sciences, University of
Memphis, Memphis TN 38152, USA; \textit{email: vnikiforv@memphis.edu}} \ and
Xiying Yuan\thanks{Corresponding author. Department of Mathematics, Shanghai
University, Shanghai 200444, China; \textit{email: xiyingyuan2007@hotmail.com
}}\thanks{Supported by NSF of China Grant No. 11101263, and by a Grant of
\textquotedblleft The First-class Discipline of Universities in
Shanghai\textquotedblright}}
\maketitle

\begin{abstract}
This paper gives tight upper bound on the largest eigenvalue $q\left(
G\right)  $ of the signless Laplacian of graphs with no paths of given order.
Thus, let $S_{n,k}$ be the join of a complete graph of order $k$ and an
independent set of order $n-k,$ and let $S_{n,k}^{+}$ be the graph obtained by
adding an edge to $S_{n,k}.$

The main result of the paper is the following theorem: $\medskip$

Let $k\geq1,$ $n\geq7k^{2},$ and let $G$ be a graph of order $n$.

\ \ \emph{(i)} if $q\left(  G\right)  \geq q\left(  S_{n,k}\right)  ,$ then
$P_{2k+2}\subset G,$ unless $G=S_{n,k};$

\ \ \emph{(ii)} if $q\left(  G\right)  \geq q\left(  S_{n,k}^{+}\right)  ,$
then $P_{2k+3}\subset G,$ unless $G=S_{n,k}^{+}.\medskip$

The main ingredient of our proof is a stability result of its own interest,
about graphs with large minimum degree and with no long paths. This result
extends previous work of Ali and Staton.\medskip



\end{abstract}

\section{Introduction}

Given a graph $G,$ the $Q$-index of $G$ is the largest eigenvalue $q\left(
G\right)  $ of its signless Laplacian $Q\left(  G\right)  $. In this paper we
determine the maximum $Q$-index of graphs with no paths of given order. This
extremal problem is related to other similar problems, so we shall start by an
introductory discussion.

In the ground-breaking paper \cite{ErGa59}, Erd\H{o}s and Gallai established
many fundamental extremal relations about graphs with no path of given order,
for example:\emph{ if }$G$\emph{ is a graph of order }$n$\emph{ with no
}$P_{k+2},$\emph{ then }$e\left(  G\right)  \leq kn/2.$ The work of Erd\H{o}s
and Gallai caused a surge of later improvements and enhancements, not
subsiding to the present day; below we mention some of these results and make
a contribution of our own.

Let $S_{n,k}$ be the join of a complete graph of order $k$ and an independent
set of order $n-k;$ i.e., $S_{n,k}=K_{k}\vee\overline{K}_{n-k}.$ Also, let
$S_{n,k}^{+}$ be the graph obtained by adding an edge to $S_{n,k}.$ Write
$\mathcal{G}\left(  n\right)  $ for the family of all graphs of order $n,$ and
$P_{l}$ for the path of order $l.$

A nice and definite enhancement of the Erd\H{o}s-Gallai result has been
obtained by Balister, Gyori, Lehel and Schelp \cite{BGLS08}.

\begin{theorem}
\label{BES} Let $k\geq1,$ $n>\left(  5k+4\right)  /2,$ $G\in\mathcal{G}\left(
n\right)  ,$ and let $G$ be connected.

\emph{(i)} if $e\left(  G\right)  \geq e\left(  S_{n,k}\right)  ,$ then
$P_{2k+2}\subset G$, unless $G=S_{n,k};$

\emph{(ii)} if $e\left(  G\right)  \geq e\left(  S_{n,k}^{+}\right)  ,$ then
$P_{2k+3}\subset G,$ unless $G=S_{n,k}^{+}.$
\end{theorem}

The main result of this paper is in the spirit of a recent trend in extremal
graph theory involving spectral parameters of graphs; most often this is the
largest eigenvalue $\mu\left(  G\right)  $ of the adjacency matrix of a graph
$G.$ The central question in this setup is the following one:\medskip

\textbf{Problem A }\emph{Given a graph }$F,$\emph{ what is the maximum }%
$\mu\left(  G\right)  $\emph{ of a graph }$G\in\mathcal{G}\left(  n\right)
$\emph{ with no subgraph isomorphic to }$F?\medskip$

Quite often, the results for $\mu\left(  G\right)  $ closely match the
corresponding edge extremal results. For illustration, compare Theorem
\ref{BES} with the following result, obtained in \cite{Nik10}.

\begin{theorem}
Let $k\geq1,$ $n\geq2^{4k+4}$ and $G\in G\left(  n\right)  .$

\emph{(i)} if $\mu\left(  G\right)  \geq\mu\left(  S_{n,k}\right)  ,$ then
$P_{2k+2}\subset G,$ unless $G=S_{n,k};$

\emph{(ii)} if $\mu\left(  G\right)  \geq\mu\left(  S_{n,k}^{+}\right)  ,$
then $P_{2k+3}\subset G,$ unless $G=S_{n,k}^{+}.$
\end{theorem}

In fact, our paper contributes to an even newer trend in extremal graph
theory, a variation of Problem A for the $Q$-index of graphs, where the
central question is the following one:\medskip

\textbf{Problem B }\emph{Given a graph }$F,$\emph{ what is the maximum }%
$Q$\emph{-index a graph }$G\in\mathcal{G}\left(  n\right)  $\emph{ with no
subgraph isomorphic to }$F?\medskip$

This question has been resolved for various subgraphs, among which are the
matchings. Thus, write $M_{k}$ for a matching of $k$ edges. In \cite{Yu08} Yu
proved the following definite result about $M_{k}.$

\begin{theorem}
\label{tYu}Let $k\geq1$ and $G\in\mathcal{G}\left(  n\right)  .$

\emph{(i) }if $2k+2\leq n<\left(  5k+3\right)  /2$ and $q\left(  G\right)
\geq4k,$ then $M_{k+1}\subset G,$ unless $G=K_{2k+1}\cup$ $\overline
{K}_{n-2k-1};$

\emph{(ii) }if $n=\left(  5k+3\right)  /2$ and $q\left(  G\right)  \geq4k,$
then $M_{k+1}\subset G,$ unless $G=K_{2k+1}\cup$ $\overline{K}_{n-2k-1}$ or
$G=S_{n,k};$

\emph{(iii) }if $n>\left(  5k+3\right)  /2$ and $q\left(  G\right)  \geq
q\left(  S_{n,k}\right)  ,$ then $M_{k+1}\subset G,$ unless $G=S_{n,k}.$
\end{theorem}

We are mostly interested in clause \emph{(iii) }of this theorem. As it turns
out, the focus on a subgraph as simple as $M_{k}$ conceals a much stronger
conclusion that can be drawn from the same premises. We arrive thus at the
main result of the present paper.

\begin{theorem}
\label{mt}Let $k\geq1,$ $n\geq7k^{2},$ and $G\in\mathcal{G}\left(  n\right)
.$

\emph{(i)} if $q\left(  G\right)  \geq q\left(  S_{n,k}\right)  ,$ then
$P_{2k+2}\subset G,$ unless $G=S_{n,k};$

\emph{(ii)} if $q\left(  G\right)  \geq q\left(  S_{n,k}^{+}\right)  ,$ then
$P_{2k+3}\subset G,$ unless $G=S_{n,k}^{+}.$
\end{theorem}

Our proof of Theorem \ref{mt} is quite complicated and builds upon several
results, among which is a stability theorem enhancing previous results by
Erd\H{o}s and Gallai and Ali and Staton. We begin with a corollary of Theorems
1.9 and 1.12 of Erd\H{o}s and Gallai \cite{ErGa59}.

\begin{theorem}
\label{tEG} Let $k\geq2,$ $G$ be a $2$-connected graph, and $u$ be a vertex of
$G.$ If $d\left(  w\right)  \geq k$ for all vertices $w\neq u,$ then $G$ has a
path of order $\min\left\{  \nu(G),2k\right\}  ,$ with end vertex $u.$
\end{theorem}

To state the next result set $L_{t,k}:=K_{1}\vee tK_{k},$ i.e., $L_{t,k}$
consists of $t$ complete graphs of order $k+1,$ all sharing a single common
vertex; call the common vertex the \emph{center }of $L_{t,k}$. In
\cite{AlSt96}, Ali and Staton gave the following stability theorem.

\begin{theorem}
\label{AS}Let $k\geq1,$ $n\geq2k+1,$ $G\in\mathcal{G}\left(  n\right)  ,$ and
$\delta\left(  G\right)  \geq k.$ If $G$ is connected, then $P_{2k+2}\subset
G,$ unless $G\subset S_{n,k},$ or $n=tk+1$ and $G=L_{t,k}$.
\end{theorem}

In the light of Theorem \ref{BES}, the theorem of Ali and Staton suggests a
possible continuation for $P_{2k+3},$ which however is somewhat more
complicated to state and prove.

\begin{theorem}
\label{AS+}Let $k\geq2,$ $n\geq2k+3,$ $G\in\mathcal{G}\left(  n\right)  $ and
$\delta\left(  G\right)  \geq k.$ If $G$ is connected, then $P_{2k+3}\subset
G,$ unless one of the following holds:

\emph{(i)} $G\subset S_{n,k}^{+}$;

\emph{(ii)} $n=tk+1$ and $G=L_{t,k};$

\emph{(iii)} $n=tk+2$ and $G\subset K_{1}\vee(\left(  t-1\right)  K_{k}\cup
K_{k+1});$

\emph{(iv) }$n=\left(  s+t\right)  k+2$ and $G\ $is obtained by joining the
centers of two disjoint graphs $L_{s,k}$ and $L_{t,k}$.
\end{theorem}

The remaining part of the paper is organized as follows. In the next section
we give the proofs of Theorems \ref{AS+} and \ref{mt}. In the concluding
remarks we round up the general discussion and state a conjecture about
further enhancement of Theorem \ref{mt}.

\section{Proofs}

For graph notation and concepts undefined here, we refer the reader to
\cite{Bol98}. For introductory material on the signless Laplacian see the
survey of Cvetkovi\'{c} \cite{C10} and its references. In particular, let $G$
be a graph, and $X$ be a set of vertices of $G.$ We write:

- $V\left(  G\right)  $ for the set of vertices of $G,$ and $e\left(
G\right)  ,\nu\left(  G\right)  $ for the number of its edges and its
vertices, respectively;

- $G\left[  X\right]  $ for the graph induced by $X,$ and $E\left(  X\right)
$ for $E\left(  G\left[  X\right]  \right)  ;$

- $\Gamma\left(  u\right)  $ for the set of neighbors of a vertex $u,$ and
$d\left(  u\right)  $ for $\left\vert \Gamma\left(  u\right)  \right\vert
.$\bigskip

\subsection{Proof of Theorem \ref{AS+}}

\begin{proof}
Assume for a contradiction that $P_{2k+3}\nsubseteq G.$ Let us first suppose
that $G$ is $2$-connected and let $C=\left(  v_{1},\ldots,v_{l}\right)  $ be a
longest cycle in $G.$ Set $V^{\prime}:=V\left(  G\right)  \backslash V\left(
C\right)  .$ A theorem of Dirac \cite{Dir52} implies that $l\geq2k,$ and
$P_{2k+3}\nsubseteq G$ implies that $l\leq2k+1$. \ As $C$ is maximal, no
vertex in $V^{\prime}$ can be joined to consecutive vertices in $C.$

Suppose first that $l=2k.$ We shall show that the set $V^{\prime}$ is
independent. Assume the opposite: let $uv$ be an edge in $V^{\prime},$ let
$C\left(  u\right)  =\Gamma(u)\cap V\left(  C\right)  $ and $C\left(
v\right)  =\Gamma(u)\cap V\left(  C\right)  .$ Since $G$ is connected,
$P_{2k+3}\nsubseteq G$ implies that $\left\vert C(u)\right\vert \geq k-1$ and
$\left\vert C(v)\right\vert \geq k-1.$

If there is a vertex $w\in C(v)\backslash C\left(  u\right)  ,$ then the
distance along $C$ between $w$ and any vertex in $C\left(  u\right)  $ is at
least $3.$ Hence $C\left(  u\right)  $ is contained in a segment of $2k-5$
consecutive vertices of $C$ and so $C\left(  u\right)  $ itself contains
consecutive vertices of $C,$ a contradiction$;$ hence $C\left(  v\right)
\subset C\left(  u\right)  $ and by symmetry we conclude that that $C\left(
u\right)  =C\left(  v\right)  .$

Finally, if $k\geq4,$ then $C(v)$ contains two vertices at distance $2$ along
$C,$\ and so $C$ can be extended, a contradiction. The remaining simple cases
$k=2$ and $3$ are left to the reader. Therefore $V^{\prime}$ is independent.

Clearly, every vertex $u\in V^{\prime}$ has exactly $k$ neighbors in $C$ and
therefore, either $\Gamma(u)=\left\{  v_{1},v_{3},\ldots,v_{2k-1}\right\}  $
or $\Gamma(u)=\left\{  v_{2},v_{4},\ldots,v_{2k}\right\}  .$ Let $u,w\in
V^{\prime},$ and assume that $\Gamma(u)=\left\{  v_{2},v_{4},\ldots
,v_{2k}\right\}  .$ If $\Gamma(w)=\left\{  v_{1},v_{3},\ldots,v_{2k-1}%
\right\}  ,$ then $C$ can be extended; hence $\Gamma(v)=\left\{  v_{2}%
,v_{4},\ldots,v_{2k}\right\}  $ for every $v\in$ $V^{\prime}.$

To complete the case $l=2k$ we shall show that $\left\{  v_{1},v_{3}%
,\ldots,v_{2k-1}\right\}  $ is independent. Assume the opposite: let $\left\{
x,y\right\}  \subset\left\{  v_{1},v_{3},\ldots,v_{2k-1}\right\}  $ and
$\left\{  x,y\right\}  \in E\left(  G\right)  $. By symmetry we can assume
that $x=v_{1}$ and $y=v_{2s+1}.$ Taking $u\in V^{\prime},$ we see that the
sequence
\[
u,v_{2},v_{3,}...,v_{2s+1},v_{1},v_{2k},v_{2k-1,}\ldots,v_{2s+2},u
\]
is a cycle longer than $C,$ a contradiction. Hence the set $\left\{
v_{1},v_{3},\ldots,v_{2k-1}\right\}  \cup V^{\prime}$ is independent and so
$G\subset S_{n,k}\subset S_{n,k}^{+}.$

Suppose now that $l=2k+1.$ Clearly $P_{2k+3}\nsubseteq G$ implies that
$V^{\prime}$ is independent. If $u,v\in V^{\prime}$ and $w\in\Gamma\left(
v\right)  \backslash\Gamma\left(  u\right)  ,$ the two neighbors of $w$ along
$C$ do not belong to $\Gamma\left(  u\right)  $ because $P_{2k+3}\nsubseteq
G.$ Hence $\Gamma\left(  u\right)  $ is a subset of $2k-2$ consecutive
vertices of $C$ and so $u$ is joined to two consecutive vertices of $C,$ a
contradiction. Hence, all vertices of $V^{\prime}$ are joined to the same set
of size $k;$ by symmetry let this set be $\left\{  v_{2},v_{4},\ldots
,v_{2k}\right\}  .$

We shall show that the set $\left\{  v_{1},v_{3},\ldots,v_{2k-1}\right\}  $ is
independent. Indeed, assume that $\left\{  v_{2s+1},v_{2t+1}\right\}  \in
E\left(  G\right)  $ and $1\leq2s+1<2t+1\leq2k-1.$ Taking $u,w\in V^{\prime},$
we see that the sequence
\[
u,v_{2s+2},v_{2s+3},\ldots,v_{2t+1},v_{2s+1},v_{2s-1},\ldots,v_{2t+2},w
\]
is a path of order $2k+3,$ contrary to our assumption$.$ Hence letting
\[
V_{2}:=\left\{  v_{1},v_{3},\ldots,v_{2k-1},v_{2k+1}\right\}  \cup V^{\prime
}\text{ \ \ and \ \ }V_{1}=V\left(  G\right)  \backslash V_{2},
\]
we find that $G\subset S_{n,k}^{+}.$ This complete the proof for $2$-connected graphs.

Finally suppose that $G$ is not $2$-connected. Let $B$ be an end-block of $G$
and $u$ be its cut vertex. Clearly, $v\left(  B\right)  \geq k+1;$ Theorem
\ref{tEG} implies that $B$ contains a path of order $\min\left\{  v\left(
B\right)  ,2k\right\}  $ with end vertex $u.$ Since there are at least two
end-blocks and $P_{2k+3}\nsubseteq G,$ there is no end-block $B$ with
$v\left(  B\right)  >k+2$ and there is at most one end-block of order $k+2.$
It is obvious that $G$ contains at most two cut vertices, otherwise we have
$P_{2k+3}\subset G.$ If $G$ contains one cut vertex, then each block of $G$ is
an end-block, and then \emph{(ii) }or\emph{ (iii)} holds. If $G$ contains two
cut vertices, then \emph{(iv) }holds, completing the proof.
\end{proof}

\subsection{Some auxiliary results}

Before going further, note that%
\[
q\left(  S_{n,k}^{+}\right)  >q\left(  S_{n,k}\right)  =\frac{n+2k-2+\sqrt
{\left(  n+2k-2\right)  ^{2}-8\left(  k^{2}-k\right)  }}{2}.
\]
For $n\geq7k^{2}$ and $k\geq2$ we also find that
\begin{align}
q\left(  S_{n,k}^{+}\right)   &  >q\left(  S_{n,k}\right)  >n+2k-2-\frac
{2\left(  k^{2}-k\right)  }{n+2k-3}\label{finb}\\
&  >n+2k-3.\label{corb}%
\end{align}

If $q\left(  G\right)  \geq q\left(  S_{n,k}\right)  $ and $k\geq2$ the
inequality of Das \cite{Das04}, implies that
\[
\frac{2e(G)}{n-1}+n-2\geq q\left(  G\right)  \geq q\left(  S_{n,k}\right)
>n+2k-2-\frac{2\left(  k^{2}-k\right)  }{n+2k-3},
\]
and so,
\begin{equation}
e(G)>k\left(  n-k\right)  . \label{Dase}%
\end{equation}

We shall also use the following bound on $q\left(  G\right)  ,$ which can be
traced back to Merris \cite{Mer98},
\begin{equation}
q\left(  G\right)  \leq\max_{u\in V\left(  G\right)  }\left\{  d\left(
u\right)  +\frac{1}{d\left(  u\right)  }\sum_{v\in\Gamma\left(  u\right)
}d\left(  v\right)  \right\}  . \label{Min}%
\end{equation}

We first determine a crucial property used throughout the proof of Theorem
\ref{mt}.

\begin{proposition}
\label{dom}Let $k\geq1,$ $n>7k^{2},$ and $G\in\mathcal{G}\left(  n\right)  .$

\emph{(i)} If $q\left(  G\right)  \geq q\left(  S_{n,k}\right)  $ and
$P_{2k+2}\nsubseteq G,$ then $\Delta\left(  G\right)  =n-1;$

\emph{(ii)} If $q\left(  G\right)  \geq q\left(  S_{n,k}^{+}\right)  $ and
$P_{2k+3}\nsubseteq G,$ then $\Delta\left(  G\right)  =n-1$.
\end{proposition}

\begin{proof}
We shall prove only \emph{(ii)}, as \emph{(i)} follows similarly. We claim
that $G$ is connected. Assume the opposite and let $G_{0}$ be a component of
$G,$ say of order $n_{0}\leq n-1,$ such that $q\left(  G_{0}\right)  =q\left(
G\right)  .$ Since $2n_{0}-2\geq q\left(  G_{0}\right)  =q\left(  G\right)
>n,$ we see that $n_{0}>\left(  5k+4\right)  /2$ and Lemma \ref{BES} implies
that $2e\left(  G_{0}\right)  \leq2kn_{0}-k^{2}-k+2;$ hence, by the inequality
of Das \cite{Das04},%
\begin{align*}
q\left(  G\right)   &  =q\left(  G_{0}\right)  \leq\frac{2e(G_{0})}{n_{0}%
-1}+n_{0}-2\leq\frac{2kn-k^{2}-3k+2}{n-2}+n-3\\
&  =n+2k-3-\frac{k^{2}-k-2}{n-2}\\
&  <n+2k-2-\frac{2\left(  k^{2}-k\right)  }{n+2k-3}\\
&  \leq q\left(  S_{n,k}\right)  .
\end{align*}
This contradiction implies that $G$ is connected.

Now, we shall prove that $\Delta\left(  G\right)  =n-1.$ Assume for a
contradiction that $\Delta\left(  G\right)  \leq n-2.$ Let $u$ be a vertex for
which the maximum in the right side of (\ref{Min}) is attained. Note that
$d\left(  u\right)  \geq2k,$ for otherwise%
\[
q\left(  G\right)  \leq d\left(  u\right)  +\frac{1}{d\left(  u\right)  }%
\sum_{v\in\Gamma\left(  u\right)  }d\left(  v\right)  \leq d\left(  u\right)
+\Delta\left(  G\right)  \leq n+2k-3<q\left(  S_{n,k}^{+}\right)  ;
\]
Furthermore, since $G$ is connected, in view of Lemma \ref{BES},
\[
\sum_{v\in\Gamma\left(  u\right)  }d\left(  v\right)  \leq2e\left(  G\right)
-\sum_{v\in V\left(  G\right)  \backslash\Gamma\left(  u\right)  }d\left(
v\right)  \leq2e\left(  S_{n,k}^{+}\right)  -n+1\leq\left(  2k-1\right)
n-k^{2}-k+3,
\]
and so
\[
q\left(  G\right)  \leq d\left(  u\right)  +\frac{\left(  2k-1\right)
n-k^{2}-k+3}{d\left(  u\right)  }.
\]
The function $f\left(  x\right)  :=x+\left(  \left(  2k-1\right)
n-k^{2}-k+3\right)  /x$ is convex in $x$ for $x>0$; hence its maximum is
attained either for $x=2k$ or for $x=n-2.$ But we see that
\[
f\left(  2k\right)  =n+2k-\frac{n+\left(  k^{2}+k\right)  -3}{2k}%
<n+2k-2-\frac{2\left(  k^{2}-k\right)  }{n+2k-3}\leq q\left(  S_{n,k}\right)
,
\]
and so,
\[
q\left(  G\right)  \leq f\left(  n-2\right)  =n+2k-3-\frac{k^{2}-3k-1}%
{n-2}<n+2k-2-\frac{2\left(  k^{2}-k\right)  }{n+2k-3}\leq q\left(
S_{n,k}\right)  .
\]
This inequality contradicts the bound (\ref{finb}), completing the proof.
\end{proof}

\begin{lemma}
\label{lind}Let $k\geq2,$ $n\geq7k^{2},$ $G\in\mathcal{G}\left(  n\right)  ,$
$e\left(  G\right)  >k\left(  n-k\right)  ,$ and $\delta\left(  G\right)  \leq
k-1.$ Suppose also that $G$ has a vertex $u$ with $d\left(  u\right)  =n-1.$
If $P_{2k+3}\nsubseteq G,$ there exists an induced subgraph $H\subset G,$ with
$\nu\left(  H\right)  \geq n-k^{2},$ $\delta\left(  H\right)  \geq k,$ and
$u\in V\left(  H\right)  .$
\end{lemma}

\begin{proof}
Define a sequence of graphs, $G_{0}\supset G_{1}\supset\cdots\supset G_{r}$
using the following procedure$.$

$G_{0}:=G;$

$i:=0$;

\textbf{while} $\delta(G_{i})<k$ \textbf{do begin}

\qquad select a vertex $v\in V(G_{i})$ with $d(v)=\delta(G_{i});$

\qquad$G_{i+1}:=G_{i}-v;$

$\qquad i:=i+1;$

\textbf{end.}

Note that the while loop must exit before $i=k^{2}$. Indeed, by $P_{2k+3}%
\nsubseteq G_{i}$ Lemma \ref{BES} implies that
\[
kn-ki-\left(  k^{2}+k\right)  /2+1\geq e(G_{i})\geq e\left(  G\right)
-i\left(  k-1\right)  >k\left(  n-k\right)  -i\left(  k-1\right)  ;
\]
hence $i<k^{2}.$ Letting $H=G_{r},$ where $r$ is the last value of the
variable $i,$ the proof is completed.
\end{proof}

\subsection{Proof of \textbf{Theorem }\ref{mt}}

\begin{proof}
[Proof of \emph{(i)}]Assume for a contradiction that $P_{2k+2}\nsubseteq G.$
By Proposition \ref{dom} $G$ has a vertex $u$ with $d\left(  u\right)  =n-1.$
If $k=1,$ then $P_{4}\nsubseteq G$\ and clearly $G=S_{n,1}.$

Let $k\geq2.$ If $\delta\left(  G\right)  \geq k,$ Lemma \ref{AS} implies that
$G\subset S_{n,k}$ or $n=kt+1$ and $G=L_{t,k}.$ The latter case cannot hold
because
\begin{equation}
q\left(  L_{t,k}\right)  \leq\max_{uv\in E\left(  L_{t,k}\right)  }\left\{
d(u)+d(u)\right\}  =n-1+k\leq n+2k-3<q\left(  S_{n,k}\right)  . \label{Lt,k}%
\end{equation}
In the first case, if $G\neq S_{n,k},$ then $q\left(  G\right)  <q\left(
S_{n,k}\right)  ,$ completing the proof. Suppose now that $\delta\left(
G\right)  \leq k-1$. By (\ref{Dase}) we have $e\left(  G\right)  >k\left(
n-k\right)  $ and then Lemma \ref{lind} implies that there exists an induced
subgraph $H$ of order $n_{1}\geq n-k^{2},$ with $\delta\left(  H\right)  \geq
k$ and $u\in V\left(  H\right)  .$ Let $H^{^{\prime}}=G\left[  V\left(
G\right)  \backslash V\left(  H\right)  \right]  $. Theorem \ref{AS} implies
that $H\subset S_{n_{1},k},$ or $n_{1}=tk+1$ and $H=L_{t,k}.$

Assume first that $n_{1}=tk+1$ and $H=L_{t,k}.$ Obviously $u$ is the center of
$H.$ Note that there is no edge between $V(H^{^{\prime}})$ and $V\left(
H\right)  \backslash\left\{  u\right\}  ,$ for otherwise $P_{2k+2}\subset G.$
Therefore,
\[
e(H^{^{\prime}})=e\left(  G\right)  -e\left(  H\right)  -\left(
n-n_{1}\right)  >k(n-k)-\frac{\left(  k+1\right)  \left(  n_{1}-1\right)  }%
{2}-\left(  n-n_{1}\right)  .
\]
After some algebra, we find that $e(H^{^{\prime}})>\frac{1}{2}\left(
k-1\right)  \left(  n-n_{1}\right)  ;$ hence $P_{k+1}\subset H^{^{\prime}}$
$($see \cite{ErGa59}$).$ Since $u$ is a dominating vertex and $P_{k+1}\subset
H,$ we see that $P_{2k+2}\subset G,$ a contradiction.

Assume now that $H\subset S_{n_{1},k}.$ Write $I$ for the independent set of
size $n_{1}-k$ of $H.$ As $\delta\left(  H\right)  \geq k,$ $H$ contains a
path $P_{2k+1}$ with both ends in $I.$ Thus, the set $V\left(  H^{\prime
}\right)  \cup I$ is independent, for otherwise $P_{2k+2}\subset G.$ Hence,
$G\subset S_{n,k}$ and so $G=S_{n,k},$ completing the proof of \emph{(i)}.

\begin{proof}
[Proof of \emph{(ii)}]Assume for a contradiction that $P_{2k+3}\nsubseteq G.$
By Proposition \ref{dom} $G$ has a vertex $u$ with $d\left(  u\right)  =n-1.$
Let $k=1.$ There is an edge in $G-u,$ for otherwise $q\left(  G\right)
<q(S_{n,1}^{+}).$ If there exist two edges in $G-u,$ then $P_{5}\subset G.$ So
$G-u$ induces exactly one edge, and $G=S_{n,1}^{+}.$
\end{proof}

Let $k\geq2.$ If $\delta\left(  G\right)  \geq k,$ in view of $\Delta\left(
G\right)  =n-1$, Theorem \ref{AS+} implies that either $G\subset S_{n,k}^{+}$
or $n=tk+1$ and $G=L_{t,k},$ or $G\subset K_{1}\vee(tK_{k}\cup K_{k+1}).$ The
inequality (\ref{Lt,k}) shows that $G\neq L_{t,k},$ and $G\subset K_{1}%
\vee(tK_{k}\cup K_{k+1})$ cannot hold because%
\begin{align*}
q\left(  K_{1}\vee(tK_{k}\cup K_{k+1})\right)   &  \leq\max_{u\in V\left(
K_{1}\vee(tK_{k}\cup K_{k+1}\right)  }\left\{  d\left(  u\right)  +\frac
{1}{d\left(  u\right)  }\sum_{v\in\Gamma\left(  u\right)  }d\left(  v\right)
\right\}  \\
&  =n+k-1-\frac{k+1}{n-1}\\
&  \leq n+2k-2-\frac{2\left(  k^{2}-k\right)  }{n+2k-3}<q\left(  S_{n,k}%
^{+}\right)  .
\end{align*}
In the first case, if $G\neq S_{n,k},$ then $q\left(  G\right)  <q\left(
S_{n,k}^{+}\right)  ,$ completing the proof. Suppose therefore that
$\delta\left(  G\right)  \leq k-1$. By (\ref{Dase}) we have $e\left(
G\right)  >k\left(  n-k\right)  $ and Lemma \ref{lind} implies that there
exists an induced subgraph $H$ of order $n_{1}\geq n-k^{2},$ with
$\delta\left(  H\right)  \geq k$ and $u\in V\left(  H\right)  .$ Theorem
\ref{AS+} implies that $H$ satisfies one of the conditions \emph{(i)-(iv).
}Since $u$ is a dominating vertex in $H$, condition \emph{(iv) }is impossible.

Next, assume that $H$ satisfies \emph{(ii) }or \emph{(iii)}$.$ Clearly,
$n_{1}\geq n-k^{2}\geq3k+2.$ Let $t$ be the number of components of $H-u;$
clearly $t\geq3$. Suppose there are two components $H_{1}$ and $H_{2}$ of
$H-u,$ with edges between $H_{1}$ and $H^{\prime}$ and between $H_{2}$ and
$H^{\prime}$. Then either $P_{2k+3}\subset G,$ or there is a cycle $C_{2k+2}$
containing $u;$ hence $P_{2k+3}\subset G$ anyway. Thus, $H-u$ has $t-1$
components that are also components of $G-u.$ Let $H_{0}$ be the remaining
component of $H-u;$ set $m=v\left(  H_{0}\right)  $ and note that $k\leq m\leq
k+1.$ Write $H^{\prime\prime}$ for the graph obtained by adding $H_{0}$ to
$H^{\prime}.$ We shall show that $e(H^{\prime\prime})>\left(  k/2\right)
v\left(  H^{\prime\prime}\right)  .$ Indeed, otherwise we have
\begin{align*}
\left(  k/2\right)  \left(  n-n_{1}+m\right)   &  \geq e(H^{\prime\prime
})=e\left(  G\right)  -e\left(  H\right)  +e\left(  H_{0}\right)  -(n-n_{1})\\
&  >k\left(  n-k\right)  -e\left(  H\right)  +e\left(  H_{0}\right)
-(n-n_{1}).
\end{align*}
Now, using the obvious inequalities
\[
e\left(  H\right)  \leq n_{1}-1+\frac{\left(  k-1\right)  \left(
n_{1}-k-1\right)  }{2}+\frac{\left(  k+1\right)  k}{2}\text{ \ \ \ \ \ and
\ \ \ \ \ }e\left(  H_{0}\right)  \geq\left(  k-1\right)  m/2,
\]
together with $m\geq k,$ $n_{1}\geq n-k^{2}$ and $n\geq7k^{2},$ we obtain a
contradiction. Hence, $e(H^{\prime\prime})>\left(  k/2\right)  v\left(
H^{\prime\prime}\right)  $ and so $P_{k+2}\subset H^{\prime\prime};$ since $u$
is a dominating vertex and $P_{k+1}\subset H,$ we get $P_{2k+3}\subset G,$
which is a contradiction.

Finally, assume that $H\subset S_{n_{1},k}^{+},$ that is to say, there exists
$I\subset V\left(  H\right)  $ of size $n_{1}-k,$ such that $I$ induces at
most one edge on $H.$ If $I$ induces precisely one edge and there are edges
between $V\left(  H^{\prime}\right)  $ and $I,$ we see that $P_{2k+3}\subset
G,$ so $V\left(  H^{\prime}\right)  \cup I$ induces at most one edge. Hence,
$G\subset S_{n,k}^{+}$ and $G=S_{n,k}^{+},$ completing the proof.

Assume now that $I$ is independent and set $J=V\left(  H\right)  \backslash
I.$ Clearly, $\delta\left(  H\right)  \geq k$ implies that every vertex of $I$
is joined to every vertex in $J;$ hence, any vertex in $I$ can be joined in
$H$ to the vertex $u$ by a path of order $2k+1.$ This implies that $V\left(
H^{\prime}\right)  \cup I$ contains no paths of order $3,$ otherwise
$P_{2k+3}\subset G;$ hence $V\left(  H^{\prime}\right)  \cup I$ induces only
isolated edges and vertices.

If $V\left(  H^{\prime}\right)  \cup I$ induces exactly one edge, we certainly
have $G\subset S_{n,k}^{+}.$ Assume now that $V\left(  H^{\prime}\right)  \cup
I$ induces two or more edges. None of these edges has a vertex in $I,$ as
otherwise, using that $u$ is dominating vertex, we can construct a $P_{2k+3}$
in $G.$ Likewise, we see that each of the ends of any edge in $H^{\prime}$ is
joined only to $u.$ We shall show that $q\left(  G\right)  <q\left(
S_{n,k}\right)  .$

Let $\left(  x_{1},\ldots,x_{n}\right)  $ be a positive unit eigenvector to
$q\left(  G\right)  .$ It is known, see, e.g., \cite{C10} that $q\left(
G\right)  =\sum_{ij\in E\left(  G\right)  }\left(  x_{i}+x_{j}\right)  ^{2}.$
Choose a vertex $v\in J\backslash\left\{  u\right\}  $ and let $ij$ be an edge
in $H^{\prime}.$ Letting $q=q\left(  G\right)  ,$ from the eigenequations for
$Q\left(  G\right)  $ we have
\[
\left(  q-2\right)  x_{i}=x_{j}+x_{u}\text{ \ \ and \ \ \ }\left(  q-2\right)
x_{j}=x_{i}+x_{u},
\]
implying that $x_{i}=x_{j}=x_{u}/\left(  q-3\right)  .$ On the other hand,
\[
\left(  q-d\left(  v\right)  \right)  x_{v}=\sum_{s\in\Gamma\left(  v\right)
}x_{s}>x_{u},
\]
implying that $x_{v}>x_{i}$ as $d\left(  v\right)  \geq\left\vert I\right\vert
\geq n-k^{2}-k>3.$

For every $ij\in E\left(  H^{\prime}\right)  ,$ remove the edge $ij$ and join
$v$ to $i$ and $j.$ Write $G^{\prime}$ for the resulting graph. Obviously
$G^{\prime}\subset S_{n,k}.$ We see that
\[
q\left(  S_{n,k}\right)  \geq q\left(  G^{\prime}\right)  \geq\sum_{ij\in
E\left(  G^{\prime}\right)  }\left(  x_{i}+x_{j}\right)  ^{2}>\sum_{ij\in
E\left(  G\right)  }\left(  x_{i}+x_{j}\right)  ^{2}=q\left(  G\right)  ,
\]
a contradiction showing that $V\left(  H^{\prime}\right)  \cup I$ induces at
most one edge and so $G\subset S_{n,k}^{+},$ completing the proof.
\end{proof}

\section{Concluding remarks}

In this paper we improve Theorem \ref{tYu} of Yu, by showing that if $G$ is a
graph of sufficiently large order $n$, then the condition $q\left(  G\right)
>q\left(  S_{n,k}\right)  $ implies that $P_{2k+2}\subset G.$ It is very
likely our own Theorem \ref{mt} can be improved in a similar way as stated in
the following conjecture for cycles.

\begin{conjecture}
Let $k\geq2$ and let $G$ be a graph of sufficiently large order $n.$ 

(i) if $q\left(  G\right)  \geq q\left(  S_{n,k}\right)  ,$ then
$C_{2k+1}\subset G,$ unless $G=S_{n,k};$

(i) if $q\left(  G\right)  \geq q\left(  S_{n,k}^{+}\right)  ,$ then
$C_{2k+2}\subset G,$ unless $G=S_{n,k}^{+}.$
\end{conjecture}

For the proof of this conjecture one may look for a stability theorem for
$2$-connected graphs with large minimum degree and with no long cycles,
similar Theorems \ref{AS} and \ref{AS+}.\ This topic is interesting by itself
and seem to have not been investigated yet.\bigskip

\textbf{Acknowledgement}

This work was done while the second author was visiting the University of Memphis.


\begin{thebibliography}{99}                                                                                               %


\bibitem {AlSt96}A.A. Ali and W. Staton, On extremal graphs with no long
paths, \emph{Electron. J. Combin.} \textbf{3} (1996), \#R20.

\bibitem {BGLS08}P.N. Balister, E. Gy\H{o}ri, J. Lehel, and R.H. Schelp,
Connected graphs without long paths, \emph{Discrete Math.} \textbf{308
}(2008), 4487--4494.

\bibitem {Bol98}B. Bollob\'{a}s, \emph{Modern Graph Theory}\textit{,} Graduate
Texts in Mathematics, 184, Springer-Verlag, New York (1998).

\bibitem {C10}D. Cvetkovi\'{c}, Spectral theory of graphs based on the
signless Laplacian, Research Report, (2010), available at:
$\emph{http://www.mi.sanu.ac.rs/projects/signless\_L\_reportApr11.pdf.}$

\bibitem {Das04}K. Das, Maximizing the sum of the squares of the degrees of a
graph, \emph{Discrete Math }\textbf{285} (2004), 57 -- 66.

\bibitem {Dir52}G.A. Dirac, Some Theorems on abstract graphs, \emph{Proc.
London Math Soc.} \textbf{2} (1952) 69-81.

\bibitem {ErGa59}P. Erd\H{o}s and T. Gallai, On maximal paths and circuits of
graphs, \emph{Acta Math. Acad. Sci. Hungar} \textbf{10} (1959), 337--356.

\bibitem {ErRe62}P. Erd\H{o}s and A. R\'{e}nyi, On a problem in the theory of
graphs, \emph{Publ. Math. Inst. Hungar. Acad. Sci.} \textbf{7A} (1962), 623--641.

\bibitem {Mer98}R. Merris, A note on Laplacian graph eigenvalues, \emph{Linear
Algebra Appl. }\textbf{295 }(1998), 33-35.

\bibitem {Nik10}V. Nikiforov, The spectral radius of graphs without paths and
cycles of specified length, \emph{Linear Algebra Appl. }\textbf{432} (2010), 2243-2256.

\bibitem {Yu08}G.Yu, On the maximal signless Laplacian spectral radius of
graphs with given matching number, \emph{Proc. Japan Acad.,} \textbf{84}, Ser.
A (2008).
\end{thebibliography}
\end{document}